\documentclass[12pt]{article}
\usepackage{amsfonts}
\usepackage{latexsym}
\usepackage{cite}
\usepackage{amsmath,amsfonts,latexsym,amssymb}
\usepackage[mathscr]{eucal}
\usepackage{cases}
\usepackage{amsthm}

\usepackage[bf,small]{caption2}
\usepackage{float}
\usepackage{graphicx}
\usepackage{amsmath}
\usepackage{amssymb}
\usepackage[all]{xy}

\newtheorem{theorem}{theorem}[section]
\newtheorem{thm}[theorem]{Theorem}
\newtheorem{lem}[theorem]{Lemma}

\newtheorem{prop}[theorem]{Proposition}
\newtheorem{cor}[theorem]{Corollary}

\newtheorem{defn}[theorem]{Definition}

\newtheorem{rmk}[theorem]{Remark}

\begin{document}

\title{\textbf{Enumerating typical abelian coverings of Cayley graphs}}
\author{\Large Haimiao Chen
\footnote{Email: \emph{chenhm@math.pku.edu.cn}}\\
\normalsize \em{School of Mathematical Science, Peking University, Beijing, China}}
\date{}
\maketitle

\begin{abstract}
In this article we complete the work of enumerating typical abelian coverings of Cayley graphs, by reducing the problem to enumerating certain subgroups of finite abelian groups.
\end{abstract}
\emph{key words:} Cayley graph, typical abelian covering, enumeration, subgroup of abelian group.   

\section{Introduction}

We consider the problem of enumerating isomorphism classes of typical abelian coverings of Cayley graphs. This problem has recently been
considered as one of the central research topics in enumerative topological graph theory and has been partially solved in consecutive papers \cite{typical3,
typical2,typical}.

First, recall some basic notions in graph theory. One can refer to \cite{Graph}.

The graphs considered in this paper are finite, simple and connected. For a graph $G$, use $V(G)$ and $E(G)$ to denote the set of vertices and edges,
respectively. The neighborhood of a vertex $v\in V(G)$, denoted by $N(v)$, is the set of vertices adjacent to $v$.

A \emph{covering }of graphs $p:\tilde{G}\rightarrow G$ is a surjection
$p:V(\tilde{G})\rightarrow V(G)$
such that
$p|_{N(\tilde{v})}:N(\tilde{v})\rightarrow N(v)$
is bijective for all $v\in V(G)$ and $\tilde{v}\in p^{-1}(v)$. We say that $p$ is \emph{regular} if $\textrm{Aut}(\tilde{G})$ acts transitively
on each fiber $p^{-1}(v)$. Two coverings $p_{i}:\tilde{G}_{i}\rightarrow G,i=1,2,$ are said to be \emph{isomorphic} if there
exists a graph isomorphism $\Phi:\tilde{G}_{1}\rightarrow\tilde{G}_{2}$ such that $p_{2}\circ\Phi=p_{1}$.

Let $A$ be a finite group and let $X$ be a subset of $A$ such that $X=X^{-1}$ and $1\notin X$. The \emph{Cayley graph
on $A$ relative to $X$}, denoted as $G=\textrm{Cay}(A,X)$, is the graph having vertex set $V(G)=A$ and edge set
$E(G)=\{\{g,gx\}\colon g\in A,x\in X\}$.
It is clear that $G$ is connected if and only if $X$ generates $A$. A \emph{circulant} graph is a Cayley graph on a cyclic group.

A \emph{typical} covering is a covering
$f_{\ast}:\textrm{Cay}(A,X)\rightarrow\textrm{Cay}(B,Y)$
induced by a surjective group homomorphism $f:A\rightarrow B$ such that $f(X)=Y$; it is regular with covering transformation group
$\ker f$. We assume $X\cap\ker f=\emptyset$ and that $f|_{X}:X\rightarrow Y$ is bijective in order to deal with simple graphs only.
The typical covering is called \emph{circulant} (\emph{abelian}) if $A$ is a cyclic (abelian) group.

Typical circulant coverings of a circulant graph were enumerated in \cite{typical3, typical2}, and typical abelian coverings
with a prime number of folds of a circulant graph were enumerated in \cite{typical}.

In this paper we get some more general results. We count typical abelian coverings of a Cayley graph on any finite abelian
group with any given abelian covering transformation group in Theorem \ref{thm:typical1}. Furthermore, we count those with any given number of folds
in Theorem \ref{thm:typical2}. Both results are given by explicit formulas. Moreover, quite general but not very concrete results are
given in Theorem \ref{thm:classfication} and Theorem \ref{thm:typical}.
These are settled by clarifying the connection between typical coverings and subgroups of some abelian groups, and then reducing the
problem to enumerating certain subgroups of abelian groups. The topics on enumerations of abelian groups are important and related
to ours, so we devote several pages to discussing them in Section 3.

Here are some conventions for notation. For a ring $\mathcal{R}$, we use $\mathcal{R}^{n,m}$ to denote the set of $n\times m$ matrices
with entries in $\mathcal{R}$ and identify $\mathcal{R}^{1,m}$ with $\mathcal{R}^{m}$; let $\textrm{GL}(m,\mathcal{R})$ denote the
set of $m\times m$ invertible matrices with entries in $\mathcal{R}$. For $M\in\mathcal{R}^{n,m}$, by $\langle M\rangle$ we mean the
subgroup of $\mathcal{R}^{m}$ generated by the row vectors of $M$. For a finite abelian group $A$, let $\exp(A)$ denote the exponent
of $A$, that is, the least positive integer $s$ such that $sa=0$ for each $a\in A$. The cyclic group $\mathbb{Z}/n\mathbb{Z}$ is
abbreviated by $\mathbb{Z}_{n}$. For a real number $x$, the least integer not less than $x$ is denoted by $\lfloor x\rfloor$.

\section{Classifying typical abelian coverings of Cayley graphs}

In this section, a connection between typical abelian coverings of a Cayley graph on an abelian group and certain subgroups
of some abelian group determined by the Cayley graph is established.

The following Lemma was proved in \cite{typical}:
\begin{lem} \label{lem:typical}
Two connected typical coverings $f_{i}:\emph{Cay}(A_{i},X_{i})\rightarrow\emph{Cay}(B,Y),\\i=1,2$, are isomorphic if and only
if there exists a group isomorphism $\phi:A_{1}\rightarrow A_{2}$ such that $f_{2}\circ\phi=f_{1}$ and $\phi(X_{1})=X_{2}$.
\end{lem}

From now on we assume that all groups are abelian, and the group operation is viewed as addition.

Let $\textrm{Cay}(B,Y)$ be a Cayley graph such that $Y=-Y$, $\langle Y\rangle=B$. Suppose further that $Y=\{\pm y_{1},\cdots,\pm y_{l}\}\cup\{y'_{1},\cdots,y'_{l'}\}$, where $2y_{i}\neq 0$ for all $y_{i}, 1\leqslant i\leqslant l$, and
$2y'_{j}=0$ for all $y'_{j}, 1\leqslant j\leqslant l'$. Then the valence of $\textrm{Cay}(B,Y)$ is $2l+l'$. Let
\begin{align}
R(Y)=\{(a_{1},\cdots,a_{l+l'})\in\mathbb{Z}^{l+l'}\colon\sum\limits_{i=1}^{l}a_{i}y_{i}+\sum\limits_{j=1}^{l'}a_{l+j}y'_{j}=0\}.
\end{align}
Then $R(Y)$ contains the subgroup $R_{0}$, where
\begin{align}
R_{0}=\{(a_{1},\cdots,a_{l+l'})\colon a_{1}=\cdots =a_{l}=0; 2|a_{l+j}, 1\leqslant j\leqslant l' \},
\end{align}
and $B\cong\mathbb{Z}^{l+l'}/R(Y)$.

For $i=1,\cdots,l+l'$, let $E_{i}=(E_{i,1},\cdots,E_{i,l+l'})\in\mathbb{Z}^{l+l'}$ with $E_{i,j}=\delta_{i,j}$. Put
\begin{align}
E=\{\pm E_{1},\cdots,\pm E_{l}\}\cup\{E_{l+1},\cdots,E_{l+l'}\}.
\end{align}
\begin{lem}
Each typical covering $f:\emph{Cay}(A,X)\rightarrow\emph{Cay}(B,Y)$ with covering transformation group $F$ is isomorphic to a
typical covering
$$\emph{Cay}(\mathbb{Z}^{l+l'}/C,q_{C}(E))\rightarrow\emph{Cay}(\mathbb{Z}^{l+l'}/R(Y),q_{R(Y)}(E))\cong\emph{Cay}(B,Y),$$
the covering induced by the canonical epimorphism $\mathbb{Z}^{l+l'}/C\rightarrow\mathbb{Z}^{l+l'}/R(Y)$,
for some $C$ with $R_{0}\leqslant C\leqslant R(Y)$ such that $R(Y)/C\cong F$, where $q_{C}:\mathbb{Z}^{l+l'}\rightarrow\mathbb{Z}^{l+l'}/C$
is the canonical quotient map.
\end{lem}
\begin{proof}
Recall the assumption in Section 1 that $X\cap\ker f=\emptyset$ and $f|_{X}:X\rightarrow Y$ is bijective. We have
$X=\{\pm x_{1},\cdots,\pm x_{l}\}\cup\{x'_{1},\cdots,x'_{l'}\}$, where $x_{i}=(f|_{X})^{-1}(y_{i}),1\leqslant i\leqslant l$,
and $x'_{j}=(f|_{X})^{-1}(y'_{j}),1\leqslant j\leqslant l'$. Let $$C=\{(a_{1},\cdots,a_{l+l'})\in\mathbb{Z}^{l+l'}\colon\sum\limits_{i=1}^{l}a_{i}x_{i}+\sum\limits_{j=1}^{l'}a_{l+j}x'_{j}=0\}.$$
Clearly $R_{0}\leqslant C\leqslant R(Y)$. The desired isomorphism $A\cong\mathbb{Z}^{l+l'}/C$ is given by sending
$x_{i}$ to $q_{C}(E_{i})$ and sending $x'_{j}$ to $q_{C}(E_{l+j})$.
\end{proof}

Observe that we have a short exact sequence $0\rightarrow F\rightarrow A\rightarrow B\rightarrow 0$ of finite abelian
groups; it follows that $\exp(A)|\exp(B)\exp(F)$. Hence
\begin{align}
R:=\{(a_{1},\cdots,a_{l+l'})\in\mathbb{Z}^{l+l'}\colon\exp(B)\exp(F)|a_{i}, 1\leqslant i\leqslant l+l'\}
\end{align}
is contained in $C$. Let
\begin{align}
\overline{R(Y)}=R(Y)/R,\hspace{4mm} \overline{C}=C/R, \hspace{4mm} \overline{R_{0}}=(R_{0}+R)/R.
\end{align}

From Lemma \ref{lem:typical} we see that two typical coverings $$\textrm{Cay}(\mathbb{Z}^{l+l'}/C_{i},q_{C_{i}}(E))\rightarrow\textrm{Cay}(\mathbb{Z}^{l+l'}/R(Y),q_{R(Y)}(E)), \hspace{4mm} i=1,2,$$
are isomorphic if and only if $C_{1}=C_{2}$ as well as $\overline{C_{1}}=\overline{C_{2}}$.
Thus we have the following classification:
\begin{thm} \label{thm:classfication}
Given abelian groups $A^{0}$ and $F$, the isomorphism classes of typical abelian coverings $\emph{Cay}(A,X)\rightarrow\emph{Cay}(B,Y)$
such that $A\cong A^{0}$ and the covering
transformation group is isomorphic to $F$ are in one-to-one correspondence with the subgroups $D$ of $\overline{R(Y)}$ such that $\overline{R_{0}}\leqslant D$,
$(\mathbb{Z}^{l+l'}/R)/D\cong A^{0}$ and $\overline{R(Y)}/D\cong F$.
\end{thm}

\section{Enumerating subgroups of finite abelian \\groups}

\subsection{Overview}

It is well-known that any finite abelian group $A$ is isomorphic to $\prod\limits_{p\in\Lambda}A_{(p)}$ for some finite set $\Lambda$
of prime numbers, each $A_{(p)}$ being a $p$-group. Call $A_{(p)}$ the \emph{$p$-primary part} of $A$.

Each abelian $p$-group $L$ is isomorphic to $\mathbb{Z}_{p^{\alpha_{1}}}\times\cdots\times\mathbb{Z}_{p^{\alpha_{n}}}$
for some partition
$\alpha=(\alpha_{1},\cdots,\alpha_{n})$
of
$|\alpha|:=\sum\limits_{i=1}^{n}\alpha_{i}$,
with
$\alpha_{1}\geqslant\cdots\geqslant \alpha_{n}\geqslant 1$.
Call $\alpha$ the \emph{type} of $L$. Let $\mathcal{A}_{p}(\alpha)$ denote the collection of abelian $p$-groups of type $\alpha$.
Regard the trivial group as a $p$-group of type $(0)$ for any $p$.
Given two partitions $\alpha=(\alpha_{1},\cdots,\alpha_{n})$
and $\beta=(\beta_{1},\cdots,\beta_{m})$,
write $\beta\subseteq\alpha$ if $m\leqslant n$ and $\beta_{i}\leqslant\alpha_{i}$ for $1\leqslant i\leqslant m$.

The problem of counting certain subgroups of a given finite abelian group has a long history.
Early in 1948, the number $\mathcal{N}_{p}(\alpha,\beta)$, for $\beta\subseteq\alpha$, of subgroups of type $\beta$ of an abelian
$p$-group of type $\alpha$ was determined by Delsarte \cite{number4}, Djubjuk \cite{number5} and  Yeh \cite{number}. Suppose
$\beta_{1}=\cdots =\beta_{m_{1}},\beta_{m_{1}+1}=\cdots =\beta_{m_{1}+m_{2}},\cdots, \beta_{m_{1}+
\cdots+m_{r-1}+1}=\cdots =\beta_{m_{1}+\cdots m_{r}}$
with
$m_{1}+\cdots m_{r}=m$,
in which case we generally write
\begin{align}
\beta=(\beta_{m_{1}}^{m_{1}},\cdots,\beta_{m_{1}+\cdots +m_{r}}^{m_{r}}) \label{eq:type}
\end{align}
for short, such that $\beta_{m_{1}}>\beta_{m_{1}+m_{2}}>\cdots >\beta_{m_{1}+\cdots +m_{r}}$, and
$\alpha_{\mu_{i}}<\beta_{i}\leqslant\alpha_{\mu_{i}-1}$
($i=1,\cdots,m$; set $\alpha_{n+1}=0$).
Then
\begin{align}
\mathcal{N}_{p}(\alpha,\beta)=p^{H}\prod\limits_{i=1}^{m}(p^{\mu_{i}-i}-1)/\prod\limits_{\eta=1}^{r}\prod
\limits_{\nu=1}^{m_{\eta}}(p^{\nu}-1), \label{eq:number1}
\end{align}
where
\begin{align}
H=\sum\limits_{i=1}^{m}(\mu_{i}-2i)(\beta_{i}-1)+\frac{1}{2}\left(\sum\limits_{i=1}^{r}m_{i}^{2}\right)-\frac{1}{2}m^{2}+
\sum\limits_{i=1}^{m}\sum\limits_{\eta=\mu_{i}}^{n}\alpha_{\eta}.
\end{align}
\begin{rmk} \label{rmk}
\rm If $\beta_{1}\leqslant\alpha_{n}$, and $K\in\mathcal{A}_{p}(\beta)$ is a subgroup of $\prod\limits_{j=1}^{n}\mathbb{Z}_{p^{\alpha_{j}}}$,
then $K$ must lie in $\{(a_{1},\cdots,a_{n})\colon p^{\beta_{1}}a_{i}=0, 1\leqslant i\leqslant n\}\cong\mathbb{Z}_{p^{\beta_{1}}}^{n}$.
This shows that
\begin{align}
\mathcal{N}_{p}(\alpha,\beta)=\mathcal{N}_{p}(\beta_{1}^{n},\beta),
\end{align}
where we have droped the parentheses, writing $\beta_{1}^n$ to mean $(\beta_{1}^{n})$.
\end{rmk}

A refined problem suggested by P. Hall is to determine the number $\mathcal{N}_{p}(\alpha,\beta,\gamma)$ of subgroups of type $\beta$
of an abelian $p$-group of type $\alpha$ which have a quotient group of type $\gamma$. A partial result was obtained in \cite{number3}.
We give a few details on this problem in Section 3.3.

As for another problem, Stehling \cite{number2} gave an expression for the number $\mathcal{N}_{p,r}(\alpha)$ of subgroups of order
$p^{r}$ of an abelian $p$-group of type $\alpha$, and derived a recurrence relation; the expression is
\begin{align}
\mathcal{N}_{p,r}(\alpha)=\sum\limits_{\beta\subseteq\alpha\colon|\beta|=r}\prod\limits_{i=1}^{\alpha_{1}}
\genfrac{[}{]}{0pt}{}{a_{i}-b_{i+1}}{b_{i}-b_{i+1}}_{p}p^{(a_{i}-b_{i})b_{i+1}}, \label{exp:total}
\end{align}
where $a_{i}$ ($b_{i}$) is the number of the integers $\alpha_{j}$ ($\beta_{j}$) with $\alpha_{j}\geqslant i$
($\beta_{j}\geqslant i$) and the Gaussian binomial coefficient
is defined as
\begin{align}
\genfrac{[}{]}{0pt}{}{k}{l}_{p}=\prod\limits_{i=1}^{l}\frac{p^{k-l+i}-1}{p^{i}-1}.
\end{align}

Finally, in other related work, significant attention has been paid to the enumeration of all subgroups of an abelian group (see \cite{unimodality, total}).

\subsection{On the structure of subgroups of abelian $p$-groups}

Fix a prime number $p$ and a positive integer $k$.

\begin{defn}
\rm For $\lambda\in\mathbb{Z}_{p^{k}}$, the \emph{$p$-degree} of $\lambda$, denoted by $\deg_{p}(\lambda)$, is the unique non-negative
integer $i$ such that $\lambda=p^{i}\cdot\chi$ with $\chi$ invertible.
\end{defn}

\begin{lem} \label{lem:subgroup}
Each subgroup $K\leqslant\mathbb{Z}_{p^{k}}^{n}$ is of the form $\langle PQ\rangle$ such that $Q\in \emph{GL}(n,\mathbb{Z}_{p^{k}})$, and $P\in\mathbb{Z}_{p^{k}}^{m,n}$,
$m\leqslant n$, $P_{i,j}=\delta_{i,j}p^{k-\beta_{i}}, k\geqslant\beta_{1}\geqslant\cdots\geqslant\beta_{m}\geqslant 1$.
Moreover, $K\cong\mathbb{Z}_{p^{\beta_{1}}}\times\cdots\times\mathbb{Z}_{p^{\beta_{m}}}$.
\end{lem}
\begin{proof}
Denote $\mathcal{R}=\mathbb{Z}_{p^{k}}$.

Take a minimal generating set $\{M_{1},\cdots,M_{l}\}$ of $K$, and suppose that $M_{i}=(M_{i,1},\cdots,M_{i,n})$. Consider the matrix $M\in\mathcal{R}^{l,n}$
with the $(i,j)$-entry $M_{i,j}$. Choose $(i_{0},j_{0})$ such that $d:=\deg_{p}(M_{i_{0},j_{0}})$ is the smallest among all the degrees
$\deg_{p}(M_{i,j})$ and suppose $M_{i_{0},j_{0}}=p^{d}\cdot\chi$. Use elementary transformations to exchange the $i_{0}$-th row
with first row and the $j_{0}$-th column with first column of $M$. Furthermore, eliminate the $(i,1)$-entry for $1<i\leqslant l$ and the
$(1,j)$-entry for $1<j\leqslant n$, and then divide the first row by $\chi$. The resulting matrix, denoted $M^{(1)}$, is equal to
$S_{1}MT_{1}$ for some $S_{1}\in\textrm{GL}(l,\mathcal{R}), T_{1}\in\textrm{GL}(n,\mathcal{R})$.

Do the same thing to the down-right $(l-1)\times(n-1)$
minor of $M^{(1)}$, and then go on. At each step, we use elementary transformations to eliminate the entries of a minor in its first row and first column, except the diagonal entry. Eventually one obtains a matrix $M^{(m)}=(S_{m}\cdots S_{1})M(T_{1}\cdots T_{m})$ such that
$m\leqslant n$, $S_{i}\in\textrm{GL}(l,\mathcal{R}), T_{i}\in\textrm{GL}(n,\mathcal{R})$,  $M^{(m)}_{i,i}=p^{s_{i}},1\leqslant i\leqslant m$, with $0\leqslant s_{1}\leqslant\cdots\leqslant s_{m}<k$, and all the other entries vanish. Let $\beta_{i}=k-s_{i}$, so that $M^{(m)}=P$. Then
$K=\langle M\rangle=\langle PQ\rangle$ with $Q=(T_{1}\cdots T_{m})^{-1}$.

Since $w\mapsto w\cdot Q$ defines an isomorphism from $\langle P\rangle$ to $\langle PQ\rangle$, we see that $K\cong\langle P\rangle\cong\mathbb{Z}_{p^{\beta_{1}}}\times\cdots\times\mathbb{Z}_{p^{\beta_{m}}}$.
\end{proof}

\begin{cor} \label{cor:quotient}
If $K$ is a subgroup of type $\beta=(\beta_{1},\cdots,\beta_{m})$ of $\mathbb{Z}_{p^{k}}^{n}$, then $\mathbb{Z}_{p^{k}}^{n}/K$
has type $k^{n}-\beta$, where
\begin{align}
k^{n}-\beta=(k^{n-m},k-\beta_{m},\cdots,k-\beta_{1}). \label{eq:minus}
\end{align}
\end{cor}
\begin{proof}
By Lemma \ref{lem:subgroup} there exists $Q\in\textrm{GL}(n,\mathbb{Z}_{p^{k}})$ such that $K=\langle PQ\rangle$, where
$P\in\mathbb{Z}_{p^{k}}^{m,n}$ with $P_{i,j}=\delta_{i,j}p^{k-\beta_{i}}$. Since $w\mapsto w\cdot Q$ defines an automorphism
of $\mathbb{Z}_{p^{k}}^{n}$, we have
\begin{align}
\mathbb{Z}_{p^{k}}^{n}/K\cong \mathbb{Z}_{p^{k}}^{n}/\langle P\rangle\cong\mathbb{Z}_{p^{k}}^{n-m}\times\mathbb{Z}_{p^{k-\beta_{m}}}\times\cdots\times
\mathbb{Z}_{p^{k-\beta_{1}}}\in\mathcal{A}_{p}(k^{n}-\beta).
\end{align}
\end{proof}

\subsection{Subgroups with prescribed quotients}

\begin{prop} \label{prop:quotient}
(a) For any partitions $\alpha,\beta,\gamma$, we have
\begin{align}
\mathcal{N}_{p}(\alpha,\beta,\gamma)=\mathcal{N}_{p}(\alpha,\gamma,\beta).
\end{align}
(b) For $L\in\mathcal{A}_{p}(\alpha)$, the number of subgroups $K$ of $L$ such that $L/K\in\mathcal{A}_{p}(\beta)$ is $\mathcal{N}_{p}(\alpha,\beta)$.
\end{prop}
\begin{proof}
(a) was proven by P. Hall in an unpublished work in the 1950s, as stated in \cite{total} (see the bottom of p. 1).

(b) follows from (a): the number of subgroups $K$ of $L$ such that $L/K\in\mathcal{A}_{p}(\beta)$ is, summing over the type
$\gamma$ of $K$, $$\sum\limits_{\gamma}\mathcal{N}_{p}(\alpha,\gamma,\beta)
=\sum\limits_{\gamma}\mathcal{N}_{p}(\alpha,\beta,\gamma)=\mathcal{N}_{p}(\alpha,\beta).$$
\end{proof}

\begin{rmk}
\rm Corollary \ref{cor:quotient} together with Proposition \ref{prop:quotient} (b) shows
\begin{align}
\mathcal{N}_{p}(k^{n},\beta)=\mathcal{N}_{p}(k^{n},k^{n}-\beta).
\end{align}
\end{rmk}

\begin{prop} \label{prop:cyclic}
Suppose $\alpha=(a_{1}^{n_{1}},a_{2}^{n_{2}})$ and $s$ is a natural number with $s\leqslant a_{1}$. Then $\mathcal{N}_{p}(\alpha,(s),\beta)$ takes nonzero values only in the following cases:
\begin{enumerate}
\item when $\beta=(a_{1}^{n_{1}-1},a_{2}^{n_{2}},(a_{1}-s)), (a_{1}^{n_{1}-1},a_{1}-s,a_{2}^{n_{2}})$ or
$(a_{1}^{n_{1}-1},(a_{1}+a_{2}-s-b),a_{2}^{n_{2}-1},(b))$ for some $b$ with $a_{2}-s<b<a_{1}-s$, \begin{align}
\mathcal{N}_{p}(\alpha,(s),\beta)=\frac{p^{n_{1}}-1}{p-1}p^{(s-1)(n_{1}+n_{2}-1)};
\end{align}
\item when $\beta=(a_{1}^{n_{1}},a_{2}^{n_{2}-1},(a_{2}-s))$,
\begin{align}
\mathcal{N}_{p}(\alpha,(s),\beta)=\frac{p^{n_{2}}-1}{p-1}p^{(s-1)(n_{2}-1)+sn_{1}}. \label{eq:number2}
\end{align}
\end{enumerate}
\end{prop}
\begin{proof}
Let $L=\mathbb{Z}_{p^{a_{1}}}\times\mathbb{Z}_{p^{a_{2}}}$.
Suppose
$u=(u^{(1)}_{1},\cdots,u^{(1)}_{n_{1}},u^{(2)}_{1},\cdots,u^{(2)}_{n_{2}})\in L$
has order $p^{s}$. For $i=1,2$, there exists $Q_{i}\in\textrm{GL}(n_{i},\mathbb{Z}_{p^{a_{i}}})$
such that
$$(u^{(i)}_{1},\cdots,u^{(i)}_{n_{i}})Q_{i}=(p^{b_{i}},0,\cdots,0).$$
Then $s=\max\{a_{1}-b_{1},a_{2}-b_{2}\}$.
Denote
$v=(p^{b_{1}},0,\cdots,0,p^{b_{2}},0,\cdots,0).$

The matrices $Q_{1},Q_{2}$ define an automorphism of $L$ by
$$(w^{(1)},w^{(2)})\mapsto(w^{(1)}Q^{(1)},w^{(2)}Q^{(2)}),$$
hence $L/\langle u\rangle\cong L/\langle v\rangle$.

There is a canonical injection
$$\iota:\mathbb{Z}_{p^{a_{2}}}\rightarrow\mathbb{Z}_{p^{a_{1}}},\hspace{5mm} \lambda\mapsto p^{a_{1}-a_{2}}\cdot\lambda.$$
It can be verified that
\begin{align*}
L/\langle v\rangle &\cong\{(w^{(1)},w^{(2)})\in L\colon p^{b_{1}}w^{(1)}_{1}+\iota(p^{b_{2}}w^{(2)}_{2})=0 \text{\ in}\
\mathbb{Z}_{p^{a_{1}}}\} \\
&\cong\mathbb{Z}_{p^{a_{1}}}^{n_{1}-1}\times\mathbb{Z}_{p^{a_{2}}}^{n_{2}-1}\times\{(\lambda_{1},\lambda_{2})
\in\mathbb{Z}_{p^{a_{1}}}\colon p^{a_{2}}\lambda_{2}=\sum\limits_{i=1}^{2}p^{b_{i}}\lambda_{i}=0\}.
\end{align*}

Here we need to consider two cases. First, assume that $a_{1}-b_{1}=s$ and $a_{2}-b_{2}<s$.

If $b_{1}\leqslant b_{2}$, then $p^{a_{2}}\lambda_{2}=\sum\limits_{i=1}^{2}p^{b_{i}}\lambda_{i}=0$ if and only if $p^{b_{1}}(\lambda_{1}+p^{b_{2}-b_{1}}\lambda_{2})=p^{a_{2}}\lambda_{2}=0$, hence
$$L/\langle u\rangle\cong\mathbb{Z}_{p^{a_{1}}}^{n_{1}-1}\times\mathbb{Z}_{p^{a_{2}}}^{n_{2}}\times\mathbb{Z}_{p^{b_{1}}}.$$
To count such cyclic subgroups $\langle u\rangle$, note that $u$ can be changed into $v$ with $b_{1}=a_{1}-s$ and
$b_{2}>a_{2}-s$ by some $(Q_{1},Q_{2})$ if and only if
\begin{align*}
\min\{\deg_{p}(u^{(1)}_{\eta})\colon 1\leqslant\eta\leqslant n_{1}\}&=a_{1}-s, \\
\min\{\deg_{p}(u^{(2)}_{\eta})\colon 1\leqslant\eta\leqslant n_{2}\}&\geqslant a_{2}-s+1;
\end{align*}
there are $p^{(s-1)(n_{1}+n_{2})}(p^{n_{1}}-1)$ choices for $u$, hence there are $$\frac{p^{(s-1)(n_{1}+n_{2}-1)}(p^{n_{1}}-1)}{p^{s}-p^{s-1}}=\frac{p^{n_{1}}-1}{p-1}p^{(s-1)(n_{1}+n_{2}-1)}$$
such cyclic subgroups.

If $b_{1}>b_{2}$, $p^{a_{2}}\lambda_{2}=\sum\limits_{i=1}^{2}p^{b_{i}}\lambda_{i}=0$ if and only if $p^{b_{2}}(p^{b_{2}-b_{1}}\lambda_{1}+\lambda_{2})=p^{a_{2}+b_{1}-b_{2}}\lambda_{1}=0$, hence
$$L/\langle u\rangle\cong\mathbb{Z}_{p^{a_{1}}}^{n_{1}-1}\times\mathbb{Z}_{p^{a_{2}}}^{n_{2}-1}
\times\mathbb{Z}_{p^{b_{2}}}\times\mathbb{Z}_{p^{a_{2}+b_{1}-b_{2}}},$$
and there are also $\frac{p^{n_{1}}-1}{p-1}p^{(s-1)(n_{1}+n_{2}-1)}$ such subgroups $\langle u\rangle$.

Now suppose that $a_{2}-b_{2}=s$. In a similar way one can deduce that
$$L/\langle u\rangle\cong\mathbb{Z}_{p^{a_{1}}}^{n_{1}}\times\mathbb{Z}_{p^{a_{2}}}^{n_{2}-1}
\times\mathbb{Z}_{p^{b_{2}}},$$
and there are $\frac{p^{n_{2}}-1}{p-1}p^{(s-1)(n_{2}-1)+sn_{1}}$ such subgroups $\langle u\rangle$.
\end{proof}

\section{Formulas for enumerating coverings}

By Theorem \ref{thm:classfication} the problem of enumerating typical coverings of Cayley graphs can be reduced to counting subgroups satisfying certain conditions. In this section we discuss enumerations of various kinds.

\begin{thm} \label{thm:typical}
Suppose $B\cong\prod\limits_{p\in\Lambda'}B_{(p)}$ with $B_{(p)}\in\mathcal{A}_{p}(\alpha(p))$, $F\cong\prod\limits_{p\in\Lambda}F_{(p)}$
with $F_{(p)}\in\mathcal{A}_{p}(\beta(p))$, and $A^{0}\cong\prod\limits_{p\in\Lambda\cup\Lambda'}A^{0}_{(p)}$ with $A^{0}_{(p)}\in\mathcal{A}_{p}(\gamma(p))$ and $\Lambda, \Lambda'$ being finite sets of prime numbers. Then the number of
typical abelian coverings
$\emph{Cay}(A,X)\rightarrow\emph{Cay}(B,Y)$ such that $A\cong A^{0}$ and the covering transformation
group is isomorphic to $F$ is
\begin{align}
&N\cdot\prod\limits_{2\neq p\in\Lambda-\Lambda'}\mathcal{N}_{p}(\beta_{1}(p)^{l+l'},\beta(p))\nonumber\\
\cdot\prod\limits_{2\neq p\in\Lambda\cap\Lambda'}&\mathcal{N}_{p}(k(p)^{l+l'}-\alpha(p),k(p)^{l+l'}-\gamma(p),\beta(p)),
\end{align}
where
\begin{align}
k(p)=\alpha_{1}(p)+\beta_{1}(p),
\end{align}
and $N$ is the number of subgroups $K$ of type $k(2)^{l+l'}-\gamma(2)$ of $\overline{R(Y)}_{(2)}$ such that
$\overline{R_{0}}\leqslant K$ and $\overline{R(Y)}_{(2)}/K\in\mathcal{A}_{2}(\beta(2))$. (Recall the notations in (1), (2), (4), and (5).)
\end{thm}
\begin{proof}
By Theorem \ref{thm:classfication} it is sufficient to count subgroups $D$ of $\overline{R(Y)}$ with
$\overline{R_{0}}\leqslant D$, $(\mathbb{Z}^{l+l'}/R)/D\cong A^{0}$ and $\overline{R(Y)}/D\cong F$.

Write $\overline{R(Y)}\cong\prod\limits_{p\in\Lambda\cup\Lambda'}\overline{R(Y)}_{(p)}$ (some $\overline{R(Y)}_{(p)}$ may be trivial).
Formally set $B_{(p)}=0$ and $\alpha(p)=(0)$ for $p\in\Lambda-\Lambda'$; set $F_{(p)}=0$ and $\beta(p)=(0)$ for $p\in\Lambda'-\Lambda$.
Note that
\begin{align}
\exp(B)=\prod\limits_{p\in\Lambda'}p^{\alpha_{1}(p)}, \hspace{5mm}
\exp(F)=\prod\limits_{p\in\Lambda}p^{\beta_{1}(p)},
\end{align}
hence
$\mathbb{Z}^{l+l'}/R\cong\prod\limits_{p\in\Lambda\cup\Lambda'}\mathbb{Z}_{p^{k(p)}}^{l+l'}$.

For each $2\neq p\in\Lambda\cup\Lambda'$, we have
$$((\mathbb{Z}^{l+l'}/R)/\overline{R_{0}})_{(p)}\cong\mathbb{Z}_{p^{k(p)}}^{l+l'}.$$
By Corollary \ref{cor:quotient},
$$((\mathbb{Z}^{l+l'}/R)/\overline{R_{0}})_{(p)}/(\overline{R(Y)}/\overline{R_{0}})_{(p)}
\cong(\mathbb{Z}^{l+l'}/R)_{(p)}/\overline{R(Y)}_{(p)}\cong B_{(p)}$$
if and only if
$(\overline{R(Y)}/\overline{R_{0}})_{(p)}\in\mathcal{A}_{p}(k(p)^{l+l'}-\alpha(p))$;
$$((\mathbb{Z}^{l+l'}/R)/\overline{R_{0}})_{(p)}/(D/\overline{R_{0}})_{p}\cong A^{0}_{(p)}$$
if and only if
$(D/\overline{R_{0}})_{(p)}\in\mathcal{A}_{p}(k(p)^{l+l'}-\gamma(p))$.

Thus the number of subgroups $K$ of $\overline{R(Y)}_{(p)}$ with $(\overline{R_{0}})_{(p)}\leqslant K$,
$\mathbb{Z}_{p^{k(p)}}^{l+l'}/K\cong A^{0}_{(p)}$ and $\overline{R(Y)}_{(p)}/K\cong F_{(p)}$ is
$$\mathcal{N}_{p}(k(p)^{l+l'}-\alpha(p),k(p)^{l+l'}-\gamma(p),\beta(p)).$$

Note that
if $p\in\Lambda'-\Lambda$, then $\beta(p)=0$ and hence
$$\mathcal{N}_{p}(k(p)^{l+l'}-\alpha(p),k(p)^{l+l'}-\gamma(p),\beta(p))=1.$$
If $p\in\Lambda-\Lambda'$, then $\alpha(p)=0, \gamma(p)=\beta(p)$ and therefore
\begin{align*}
&\mathcal{N}_{p}(k(p)^{l+l'}-\alpha(p),k(p)^{l+l'}-\gamma(p),\beta(p)) \\
=\ &\mathcal{N}_{p}(\beta_{1}(p)^{l+l'},\beta_{1}(p)^{l+l'}-\beta(p),\beta(p)) \\
=\ &\mathcal{N}_{p}(\beta_{1}(p)^{l+l'},\beta(p))
\end{align*}
which follows from Corollary \ref{cor:quotient}.

Finally we emphasize that the case $p=2$ can be proven similarly.
\end{proof}

\begin{rmk}
\rm The main result of \cite{typical2} (Theorem 13) is a direct consequence of Theorem \ref{thm:typical}.
We re-derive it as follows.

Suppose $A^{0}$ is cyclic, so that $F\cong\prod\limits_{p\in\Lambda}\mathbb{Z}_{p^{\beta_{1}(p)}}, B\cong\prod\limits_{p\in\Lambda'}\mathbb{Z}_{p^{\alpha_{1}(p)}}$,
and $\gamma(p)=k(p)$. Suppose the valence of $\textrm{Cay}(B,Y)$ is $d$.
Suppose the valence of $\textrm{Cay}(B,Y)$ is $d$.

If $d$ is odd, then $l'=1$, $\alpha_{1}(2)\geqslant 1$, and
$$\overline{R_{0}}=\{w_{1},\cdots,w_{l+1}\in\mathbb{Z}_{2^{k(2)}}^{l+1}\colon w_{1}=\cdots =w_{l}=0,2|w_{l+1}\}.$$
If $K\leqslant\overline{R(Y)}_{(2)}$ has type $k(2)^{l+l'}-\gamma(2)=k(2)^{l}$ and contains $\overline{R_{0}}$, then
$(0,\cdots,0,2)=2w$ for some $w=(w_{1},\cdots,w_{l+1})\in K$.
We have
$$\deg_{2}(w_{l+1})=0; \hspace{5mm} \deg_{2}(w_{i})\geqslant k(2)-1, \hspace{5mm} 1\leqslant i\leqslant l.$$
When $\beta_{1}(2)\geqslant 1$, it is impossible that $w\in\overline{R(Y)}_{(2)}$ and hence $N=0$. If $\beta_{1}(2)=0$, then it is clear that $N=1$.

If $d$ is even, then $l'=0$, and $N=\mathcal{N}_{2}(k(2)^{l+l'}-\alpha(2),k(2)^{l+l'}-\gamma(2),\beta(2))$.

For $p\in\Lambda\cap\Lambda'$,
\begin{align*}
&\mathcal{N}_{p}(k(p)^{l+l'}-\alpha(p),k(p)^{l+l'}-\gamma(p),\beta(p)) \nonumber\\
=\ &\mathcal{N}_{p}(k(p)^{l+l'}-\alpha_{1}(p),k(p)^{l+l'-1},\beta_{1}(p))
\nonumber \\
=\ &\mathcal{N}_{p}(k(p)^{l+l'}-\alpha_{1}(p),\beta_{1}(p),k(p)^{l+l'-1})  \nonumber \\
=\ &p^{\beta_{1}(p)(l+l'-1)}
\end{align*}
which follows from Proposition \ref{prop:quotient} (a) and (\ref{eq:number2}).

For $p\in\Lambda-\Lambda'$,
$$\mathcal{N}_{p}(k(p)^{l+l'},k(p))=p^{(\beta_{1}(p)-1)(l+l'-1)}(p^{l+l'}-1).$$

Summarizing, the number of $(\prod\limits_{p\in\Lambda}p^{\beta_{1}(p)})$-fold typical circulant coverings of $\textrm{Cay}(B,Y)$ is
$\left\{
    \begin{array}{ll}
     0 \hspace{5mm} \text{\ if\ }d \text{\ is\ odd\ and\ $\beta_{1}(2)\geqslant 1$\ }\\
     \prod\limits_{p\in\Lambda}N(p) \hspace{5mm} \text{\ otherwise},
    \end{array}
    \right.$ \\
where
\begin{align}N(p)=\left\{
    \begin{array}{ll}
     p^{\beta_{1}(p)(\lfloor\frac{d}{2}\rfloor-1)}, p\in\Lambda', \\
     p^{(\beta_{1}(p)-1)(\lfloor\frac{d}{2}\rfloor-1)}(p^{\lfloor\frac{d}{2}\rfloor}-1), p\notin\Lambda'.
    \end{array}
    \right.
    \end{align}
\end{rmk}

\begin{thm} \label{thm:typical1}
Suppose $F\cong\prod\limits_{p\in\Lambda}F_{(p)}$, where $\Lambda$ is a finite set of prime numbers and
$F_{(p)}\in\mathcal{A}_{p}(\beta(p))$ for all $p\in\Lambda$. Suppose 
\begin{align*}
\beta(p)=(\beta_{1}(p),\cdots,\beta_{m(p)}(p))=(\overline{\beta}_{1}(p)^{m_{1}(p)},\cdots,\overline{\beta}_{r(p)}(p)^{m_{r(p)}(p)}),
\end{align*}
such that $\overline{\beta}_{1}(p)>\cdots >\overline{\beta}_{r(p)}(p)$.
Then the number of typical coverings of $\emph{Cay}(B,Y)$ with covering transformation group isomorphic to $F$ is
\begin{align}
2^{H'(2)}\cdot\prod\limits_{i=1}^{m'(2)}(2^{l+1-i}-1)\cdot\prod\limits_{j=m'(2)+1}^{m(2)}(2^{l+l_{0}+1-j}-1)\cdot
\prod\limits_{t=1}^{r(2)}\prod\limits_{\nu=1}^{m_{t}(2)}(2^{\nu}-1)^{-1} \nonumber\\
\cdot\prod\limits_{2\neq p\in\Lambda}\left(p^{H(p)}\cdot\prod\limits_{i=1}^{m(p)}(p^{l+l'+1-i}-1)/(\prod\limits_{t=1}^{r(p)}\prod
\limits_{\nu=1}^{m_{t}(p)}(p^{\nu}-1))\right), \label{exp:typical}
\end{align}
where $l_{0}$ is the rank of $R'(Y)$, with
\begin{align}
R'(Y)=\{(u_{1},\cdots,u_{l'})\in\mathbb{Z}_{2}^{l'}\colon\sum\limits_{j=1}^{l'}u_{j}y'_{j}=0\},
\end{align}
$m'(2)$ is the largest $\eta$ such that $\beta_{\eta}(2)>1$,
\begin{align}
H'(2)=\sum\limits_{i=1}^{m'(2)}(l+1-2i)(\beta_{i}(2)-1)+\frac{1}{2}\left(\sum\limits_{i=1}^{r(2)}m_{i}(2)^{2}\right)-\frac{1}{2}m(2)^{2}+l_{0}m'(2),
\end{align}
and for $p\neq 2$,
\begin{align}
H(p)=\sum\limits_{i=1}^{m(p)}(l+l'+1-2i)(\beta_{i}(p)-1)+\frac{1}{2}\left(\sum\limits_{i=1}^{r(p)}m_{i}(p)^{2}\right)-\frac{1}{2}m(p)^{2}.
\end{align}
\end{thm}
\begin{proof}
For $p\neq 2$, Theorem \ref{thm:typical} implies that it suffices to compute
$$\sum\limits_{\gamma(p)}\mathcal{N}_{p}(k(p)^{l+l'}-\alpha(p),k(p)^{l+l'}-\gamma(p),\beta(p))
=\mathcal{N}_{p}(k(p)^{l+l'}-\alpha(p),\beta(p)),$$
which is equal to $\mathcal{N}_{p}(\beta_{1}(p)^{l+l'},\beta(p))$ by Remark \ref{rmk}.
Then apply (\ref{eq:number1}).

For $p=2$, let $N(\gamma(2))$ be the number of the subgroups $K$ of type $k(2)^{l+l'}-\gamma(2)$ of
$\overline{R(Y)}_{(2)}$ such that $\overline{R_{0}}\leqslant K$ and $\overline{R(Y)}_{(2)}/K\in\mathcal{A}_{2}(\beta(2))$.
By Proposition \ref{prop:quotient} (b), we have
$$\mathcal{N}_{2}(\tau,\beta(2))=\sum\limits_{\gamma(2)}N(\gamma(2)),$$
where $\tau$ is the type of $\overline{R(Y)}_{(2)}/\overline{R_{0}}$.

Note that
\begin{align}
(\mathbb{Z}^{l+l'}/R)_{(2)}\cong\mathbb{Z}_{2^{k(2)}}^{l+l'},
\end{align}
where $2^{k(2)}=\exp(B_{(2)})\cdot\exp(F_{(2)})$.
Via this isomorphism $\overline{R_{0}}=(\overline{R_{0}})_{(2)}$ is identified with
\begin{align}
\{(w_{1},\cdots,w_{l+l'})\in\mathbb{Z}_{2^{k(2)}}^{l+l'}\colon w_{1}=\cdots =w_{l'}=0; 2|w_{l+j}, 1\leqslant j\leqslant l'\}
\end{align}
which is isomorphic to $\mathbb{Z}_{2^{k(2)-1}}^{l'}$.

Now if $K'\leqslant\mathbb{Z}_{2^{k(2)}}^{l+l'}/\overline{R_{0}}\cong\mathbb{Z}_{2^{k(2)}}^{l}\times\mathbb{Z}_{2}^{l'}$
has type $\beta(2)$, and
$w=(w_{1},\cdots,w_{l+l'})\\ \in K'$, with $w_{1},\cdots,w_{l}\in\mathbb{Z}_{2^{k(2)}}$ and $w_{l+1},\cdots,w_{l+l'}\in\mathbb{Z}_{2}$,
then
$$\deg_{2}(w_{i})\geqslant\alpha_{1}(2) \text{\ for\ all\ }w_{i}, 1\leqslant i\leqslant l,$$
so $(w_{1},\cdots,w_{l},0,\cdots,0)\in\overline{R(Y)}_{(2)}/\overline{R_{0}}$. Hence $w\in\overline{R(Y)}_{(2)}/\overline{R_{0}}$
if and only if $(w_{l+1},\cdots,w_{l+l'})\in R'(Y)$. This shows
$$\mathcal{N}_{2}(\tau,\beta(2))=\mathcal{N}_{2}((k(2)^{l},1^{l_{0}}),\beta(2)),$$
which is equal to $\mathcal{N}_{2}((\beta_{1}(2)^{l},1^{l_{0}}),\beta(2))$, because a subgroup $K$ of type $\beta(2)$ of $\mathbb{Z}_{2^{k(2)}}^{l}\times\mathbb{Z}_{2}^{l_{0}}$ must be contained in
$$\{(w_{1},\cdots,w_{l+l_{0}})\colon 2^{\beta_{1}(2)}w_{i}=0, 1\leqslant i\leqslant l, 2w_{j}=0 \text{\ for\ all\ }j>l\}\cong\mathbb{Z}_{2^{\beta_{1}(2)}}^{l}\times\mathbb{Z}_{2}^{l_{0}}.$$
Then apply (\ref{eq:number1}).
\end{proof}

\begin{rmk}
\rm Consider the special case when $F$ is cyclic: $F\cong\mathbb{Z}_{f}$ with
$f=\prod\limits_{p\in\Lambda}p^{\beta_{1}(p)}$. Then (\ref{exp:typical}) reduces to
\begin{align}
N'(2)\cdot\prod\limits_{2\neq p\in\Lambda}\left(p^{(l+l'-1)(\beta_{1}(p)-1)}\cdot\frac{p^{l+l'}-1}{p-1}\right),
\end{align}
where $N'(2)=1$ if $2\notin\Lambda$ and
$$N'(2)=\left\{
    \begin{array}{ll}
     2^{(l-1)(\beta_{1}(2)-1)+l_{0}}(2^{l}-1), \hspace{5mm} \beta_{1}(2)>1 \\
     2^{l+l_{0}}-1, \hspace{5mm} \beta_{1}(2)=1
    \end{array}
    \right.$$
if $2\in\Lambda$.
If $f$ is prime and $B$ is cyclic, then $l'\in\{0,1\}$ and the result specializes to Theorem 1 of \cite{typical}.
\end{rmk}

\begin{thm}  \label{thm:typical2}
Suppose $q=\prod\limits_{p\in\Lambda}p^{s(p)}$. The number of $q$-fold typical abelian coverings of $\emph{Cay}(B,Y)$ is
\begin{align}
\left(\sum\limits_{k=\lfloor\frac{s(2)-l_{0}}{l}\rfloor}^{s(2)}\sum\limits_{0\leqslant b_{k}\leqslant\cdots\leqslant b_{1}\leqslant l+l_{0} \atop b_{1}+\cdots +b_{k}=s(2)}
\genfrac{[}{]}{0pt}{}{l+l_{0}-b_{2} }{b_{1}-b_{2}}_{2}2^{(l+l_{0}-b_{1})b_{2}}\cdot\prod\limits_{i=2}^{k}\genfrac{[}{]}{0pt}{}{l-b_{i+1} }{ b_{i}-b_{i+1}}_{2}
2^{(l-b_{i})b_{i+1}}\right) \nonumber \\
\cdot\prod\limits_{2\neq p\in\Lambda}\left(\sum\limits_{k=\lfloor\frac{s(p)}{l+l'}\rfloor}^{s(p)}\sum\limits_{0\leqslant b_{k}\leqslant\cdots\leqslant b_{1}\leqslant l+l' \atop b_{1}+\cdots +b_{k}=s(p)}\prod\limits_{i=1}^{k}\genfrac{[}{]}{0pt}{}{l+l'-b_{i+1}}{b_{i}-b_{i+1}}_{p}
p^{(l+l'-b_{i})b_{i+1}}\right). \label{exp:typical2}
\end{align}
\end{thm}
\begin{proof}
By Theorem \ref{thm:typical1} and its proof, the number of typical coverings of $\textrm{Cay}(B,Y)$ with covering transformation group isomorphic to $\prod\limits_{p\in\Lambda}F_{(p)}$ with $F_{(p)}\in\mathcal{A}_{p}(\beta(p))$ is
$$\mathcal{N}_{2}((\beta_{1}(2)^{l},1^{l_{0}}),\beta(2))\cdot\prod\limits_{2\neq p\in\Lambda}\mathcal{N}_{p}(\beta_{1}(p)^{l+l'},\beta(p)).$$
So the number of $q$-fold typical coverings is
\begin{align*}
\left(\sum\limits_{\beta(2)\leqslant(s(2)^{l},1^{l_{0}}) \atop  |\beta(2)|=s(2)}\mathcal{N}_{2}((\beta_{1}(2)^{l},1^{l_{0}},\beta(2))\right)\cdot\prod\limits_{2\neq p\in\Lambda}\left(\sum\limits_{\beta(p)\leqslant s(p)^{l+l'} \atop |\beta(p)|=s(p)}\mathcal{N}_{p}(\beta_{1}(p)^{l+l'},\beta(p))\right) \\
=\left(\sum\limits_{\beta_{1}(2)=\lfloor\frac{s(2)-l_{0}}{l}\rfloor}^{s(2)}\mathcal{N}_{2,s(2)}((\beta_{1}(2)^{l},1^{l_{0}}))\right)
\cdot\prod\limits_{2\neq p\in\Lambda}\left(\sum\limits_{\beta_{1}(p)=\lfloor\frac{s(p)}{l+l'}\rfloor}^{s(p)}\mathcal{N}_{p,s(p)}(\beta_{1}(p)^{l})\right).
\end{align*}
The result can be proven by applying formula (\ref{exp:total}).
\end{proof}

\end{document}